\documentclass[12pt,a4paper]{article}

\usepackage{amsthm,amsmath,amssymb,url,comment,setspace,mleftright}
\usepackage[dvipdfm]{graphicx,color}

\allowdisplaybreaks

\theoremstyle{plain}   
\newtheorem{df}{Definition}[section]

\newtheorem{proposition}[df]{Proposition}
\newtheorem{theorem}[df]{Theorem}
\newtheorem{remark}[df]{Remark}
\newtheorem{definition}[df]{Definition}
\newtheorem{lemma}[df]{Lemma}

\theoremstyle{definition}

\theoremstyle{remark}

\makeatletter

\@addtoreset{equation}{section}
\makeatother

\numberwithin{equation}{section}  

\begin{document}

\title{Rate of convergence in first-passage percolation under low moments} 
\author{Michael Damron \\ \small Georgia Tech, Indiana University \and Naoki Kubota \\ \small Nihon University}
\date{}

\maketitle

\begin{abstract}
We consider first-passage percolation on the $d$ dimensional cubic lattice for $d \geq 2$;
that is, we assign independently to each edge $e$ a nonnegative random weight $t_e$
with a common distribution and consider the induced random graph distance (the passage time), $T(x,y)$. It is known that for each $x \in \mathbb{Z}^d$, $\mu(x) = \lim_n T(0,nx)/n$ exists and that $0 \leq \mathbb{E}T(0,x) - \mu(x) \leq C\|x\|_1^{1/2}\log \|x\|_1$ under the condition $\mathbb{E}e^{\alpha t_e}<\infty$ for some $\alpha>0$. By combining tools from concentration of measure with Alexander's methods, we show how such bounds can be extended to $t_e$'s with distributions that have only low moments. For such edge-weights, we obtain an improved bound $C (\|x\|_1 \log \|x\|_1)^{1/2}$ and bounds on the rate of convergence to the limit shape.
\end{abstract}

\section{Introduction}

\subsection{The model}
Let $d \geq 2$.
Denote the set of nearest-neighbor edges of $\mathbb{Z}^d$ by $\mathcal{E}^d$,
and let $(t_e)_{e \in \mathcal{E}^d}$ be a collection of non-negative random variables indexed by $\mathcal{E}^d$. For $x,y \in \mathbb{Z}^d$, define the passage time
\[
\tau(x,y) = \inf_{\gamma : x \to y} \tau(\gamma)\ ,
\]
where $\tau(\gamma) = \sum_{e \in \gamma} t_e$ and $\gamma$ is any lattice path from $x$ to $y$.

We will assume that $\mathbb{P}$, the distribution of $(t_e)$, is a product measure
satisfying the following conditions:
\begin{itemize}
 \item[\bf (A1)]
  $\mathbb{E}Y^2 < \infty$,
  where $Y$ is the minimum of $d$ i.i.d.~copies of $t_e$.
 \item[\bf (A2)]
  $\mathbb{P}(t_e=0) < p_c$,
  where $p_c$ is the threshold for $d$-dimensional bond percolation.
\end{itemize}
We now comment on assumptions (A1) and (A2).
From Lemma~3.1 of \cite{CD},
(A1) guarantees that $\mathbb{E} \tau(0,y)^{4-\eta}<\infty$ for all $\eta>0$ and $y \in \mathbb{Z}^d$.
Conversely, it is sufficient for (A1) that
$\mathbb{E} t_e^{(2+\epsilon)/d}$ is finite for some $\epsilon>0$.
On the other hand, (A2) ensures that
\begin{equation}\label{eq:geodesics}
\mathbb{P}(\exists \text{ a geodesic from } x \text{ to } y) = 1 \text{ for all } x,y \in \mathbb{Z}^d\ ,
\end{equation}
where a geodesic is a path $\gamma$ from $x$ to $y$ that has $\tau(\gamma) = \tau(x,y)$.
Under assumptions (A1) and (A2), equation
(1.13) and Theorem~{1.15} of \cite{Kes86} show that
there exists a norm $\mu (\cdot)$ on $\mathbb{R}^d$, which is called the time constant, such that
for $x \in \mathbb{Z}^d$, $\mathbb{P}$- almost surely,
\begin{align}
 \lim_{n \to \infty}\frac{1}{n}\tau (0,nx)
 = \lim_{n \to \infty}\frac{1}{n}\mathbb{E} \tau (0,nx)
 = \inf_{n \geq 1}\frac{1}{n}\mathbb{E} \tau (0,nx)
 = \mu (x).
 \label{eq:t-c}
\end{align}
If (A1) is replaced by the condition that
the minimum of $2d$ i.i.d.\,copies of $t_e$ has finite $d$-th moment,
then the \emph{shape theorem} holds; that is,
for all $\epsilon>0$, with probability one,
\begin{equation}\label{eq: shape_theorem}
 (1-\epsilon )B_0 \subset \frac{B(t)}{t} \subset (1+\epsilon )B_0 \text{ for all large }t\ ,
\end{equation}
where
\begin{align*}
 B(t):=\biggl\{ x+h ;x \in \mathbb{Z}^d,\,\tau (0,x) \leq t,\,h \in \Bigl[ -\frac{1}{2},\frac{1}{2} \Bigr]^d \biggr\}
\end{align*}
and $B_0:= \{ x \in \mathbb{R}^d;\mu (x) \leq 1 \}$, the \emph{limit shape}.

\subsection{Main results}
Set $T = \tau(0,x)$. Our results below consist of (a) a bound on the deviation of $\mathbb{E}T$ from $\mu$ under (A1) and (A2) and (b) outer bounds on the rate of convergence to the limit shape under these same conditions and inner bounds under stronger conditions. These should be compared to the results of Alexander \cite{Ale97}, who proved the first two with $\log$ in place of $\sqrt{\log}$ under exponential moments for $t_e$.

\begin{proposition}\label{prop:fluct}
Assume (A1) and (A2).
There exists $C_1$ such that for all $x \in \mathbb{Z}^d$ with $\|x\|_1>1$,
\begin{align*}
\mu(x) \leq \mathbb{E} T \leq \mu (x)+C_1(\|x\|_1\log \|x\|_1)^{1/2}.
\end{align*}
\end{proposition}

\begin{proof}
This result follows by directly combining the Gaussian concentration inequality we derive below in Theorem~\ref{thm:big_one} with Alexander's method of approximation of subadditive functions \cite{Ale97}. Our final bound is slightly better than the one given by Alexander in \cite{Ale97} because he used only an exponential concentration inequality. The interested reader can see the arXiv version of this paper \cite{DK_arxiv} (version 1).
\end{proof}

\begin{theorem}\label{thm:shape}
Assume (A1) and (A2).
There exists $C_2$ such that with probability one,
\begin{align}
 \frac{B(t)}{t}
 \subset \bigl\{ 1+C_2t^{-1/2} (\log t)^{1/2} \bigr\} B_0 \text{ for all large }t\ .
 \label{eq:outer}
\end{align}
If $\mathbb{E} t_e^\alpha<\infty$ for some $\alpha > 1+1/d$,
then there is $C_3$ such that with probability one,
\begin{align}
 \bigl\{ 1-C_3t^{-1/2} (\log t)^4 \bigr\} B_0 
 \subset \frac{B(t)}{t} \text{ for all large }t\ .
 \label{eq:inner}
\end{align}
\end{theorem}
The proof of Theorem~\ref{thm:shape} will be given in Section~\ref{sec:shape}.

\subsection{Relation to previous works}

The question we address here is to determine the minimal moment condition for $t_e$ that will guarantee that the passage time $\tau(0,x)$ is bounded by $\mu(x) + O(\|x\|_1^a)$ for some $a<1$; that is, that the deviation from the norm $\mu(\cdot)$ is sub-linear by a power of $\|x\|_1$. This line of work began with Cox and Durrett \cite{CD}, who found optimal conditions for existence of a shape theorem. They showed that if $Z$ is the minimum of $2d$ i.i.d.\,random variables distributed as $t_e$, then for each $x \in \mathbb{Z}^d$,
\[
\mathbb{E}Z < \infty \Leftrightarrow \tau (0,nx)/n \text{ converges a.s.}
\]
Furthermore
\[
\mathbb{E}Z^d < \infty \Leftrightarrow \text{the shape theorem }\eqref{eq: shape_theorem} \text{ holds}\ .
\]
These conditions do not address the convergence rate, and it is conceivable that the rate gets worse as the weight distribution gets closer to violating the above conditions.

The behavior $\tau(0,x) = \mu(x) + O(\|x\|_1^a)$ is important because it allows a good approximation of $\tau$ by $\mu$, and this is often useful when one knows information about the asymptotic shape. For example, if $x/\mu(x)$ is an exposed point on the boundary of the limit shape, then one can show that geodesics from $0$ to $nx$ stay within Euclidean distance $o(n)$ of the straight line connecting $0$ and $nx$ for all large $n$. However this is only possible if the error exponent $a<1$. Similar geodesic concentration was needed, for example, in the $\log n$ lower bound for the variance for the passage time in $d=2$ for certain atomic distributions \cite[Proposition~2]{ADdiff}. In that work, exponential moments were assumed for $t_e$ precisely to guarantee that $a<1$ (and this was relaxed in \cite{kubota}).

All work on the convergence rate until recently has assumed exponential moments: $\mathbb{E}e^{\alpha t_e}<\infty$ for some $\alpha>0$. Under this condition and (A2), Alexander \cite{Ale97} showed in 1997, building on work of Kesten \cite{Kes93}, that with probability one, for all large $t$,
\[
(1-Ct^{-1/2}\log t)B_0 \subset B(t)/t \subset (1+Ct^{-1/2}\log t)B_0\ .
\]
The proof of this result combined Kesten's exponential concentration inequality for $\tau(0,x)$, along with Alexander's bound on $\mathbb{E}\tau(0,x) - \mu(x)$ and direction-independent estimates: there exists $C$ such that for all large $x \in \mathbb{Z}^d$,
\begin{equation}\label{eq: nonrandom}
 \mathbb{E}\tau(0,x) -\mu(x) \leq C\| x \|_1^{1/2}\log \| x \|_1.
\end{equation}
The reason these theorems were unattainable under weaker moment conditions is the lack of available concentration inequalities for the passage time $\tau (0,x)$.

Zhang \cite{Zha10} was one of the first to establish forms of concentration and rate of convergence for $\tau(0,nx)/n$ under only existence of $m$ moments for $t_e$. He obtained estimates for central moments for $\tau(0,x)$ and a bound for the left side of \eqref{eq: nonrandom}, but only for $x = ne_1$, a multiple of the first coordinate vector, under the condition $\mathbb{E}t_e^{1+\eta}<\infty$ for some $\eta>0$. The central moment bound was improved by Chatterjee-Dey in \cite[Proposition~5.1]{CDey} by removing logarithmic factors under the same moment assumption. These works do not imply polynomial rates of convergence for the shape theorem because the bound of Zhang on the difference $\mathbb{E}\tau(0,x) - \mu(x)$ is only valid in the coordinate directions (due to use of a reflection argument), and direction-independent estimates are needed.

The first bounds for the left side of \eqref{eq: nonrandom} without an exponential moment assumption are due to Kubota \cite[Theorem~1.2]{kubota},
who showed that if $\mathbb{E}t_e^\alpha<\infty$ for some $\alpha>1$ and (A2) holds, then one has a bound for the left side of \eqref{eq: nonrandom} of $C\| x \|_1^{1-1/(6d+12)}(\log \| x \|_1)^{1/3}$. The present work grew out of attempts to improve this inequality. Our Proposition~\ref{prop:fluct} does so both in the moment assumption and in the rate of convergence. We only require $(2/d)+\epsilon$ moments for $t_e$ (note that this exponent approaches 0 as $d \to \infty$), whereas Kubota needed $1+\epsilon$. Our proofs rely on a concentration inequality which is derived under low moments using a block tensorization procedure (see Lemma~\ref{lem:ent_box}).

From our inequality for \eqref{eq: nonrandom}, we conclude a strong rate of convergence in the shape theorem under low moments. In Theorem~\ref{thm:shape}, we decrease Alexander's outer bound of the shape to
$t^{-1/2}(\log t)^{1/2}$ with only assumptions (A1) and (A2). For the inner bound of the shape, we obtain a rate of $t^{-1/2}(\log t)^4$ under existence of $\alpha$ moments for some $\alpha > 1+1/d$ and (A2). Here the low moment condition of $t_e$ plays a big role, since high edge-weights will prevent certain vertices from entering $B(t)$. The condition on $\alpha$ above comes from optimizing large-deviation bounds from recent work of Ahlberg \cite{Ahl}.

\begin{remark}
After this work was completed, \cite{DHS} obtained sub-diffusive concentration inequalities for $\tau(0,x)$, extending work of Bena\"im-Rossignol \cite{BR} to general distributions. They proved, in particular, if $\mathbb{E}t_e^2(\log t_e)_+<\infty$ for $d \geq 2$, then
\begin{equation}\label{eq: DHS}
\mathbb{P}\left( \tau(0,x) < \mathbb{E}\tau(0,x) - \lambda \sqrt{\frac{\|x\|_1}{\log \|x\|_1}} \right) \leq e^{-c\lambda} \text{ for } \|x\|_1>1,~ \lambda \geq 0\ .
\end{equation}
\eqref{eq: DHS} may be used in Alexander's method to obtain our Proposition~\ref{prop:fluct}, but with a worse moment condition. One may ask if it can be used to prove a sub-gaussian bound $\mathbb{E}\tau(0,x) \leq \mu(x) + o(\sqrt{\|x\|_1})$, as in the case of directed polymers \cite{AZ}, but no such theorem exists in an undirected model. The main complication arises from the extra logarithmic factor coming from Alexander's method, which negates the gain in the scale from \eqref{eq: DHS}.
\end{remark}

\section{Concentration inequality}\label{sec: entropy}

Our first task will be to prove the following lower-tail concentration result.
\begin{theorem}\label{thm:big_one}
Assume (A1) and (A2).
There exists $C_1>0$ such that
\begin{equation}\label{eq: main_result}
\mathbb{P}\left( T - \mathbb{E}T \leq -t \sqrt{\|x\|_1} \right) \leq e^{-C_1t^2} \text{ for } t \geq 0,~ x \in \mathbb{Z}^d\ .
\end{equation}
\end{theorem}
This theorem extends Talagrand's \cite[Theorem~8.2.3]{Talagrand} concentration inequality, which states that \eqref{eq: main_result} holds under assumptions (A2) and $\mathbb{E}t_e^2<\infty$. To derive Theorem~\ref{thm:big_one},
we use the entropy method of Bucheron-Lugosi-Massart (see \cite{BLM} for an introduction). 

\begin{definition}
If $X$ is a non-negative random variable with $\mathbb{E}X<\infty$ then the entropy of $X$ is defined as $Ent~X = \mathbb{E}X\log X - \mathbb{E}X \log \mathbb{E}X$.
\end{definition}

Theorem~\ref{thm:big_one} is a straightforward consequence of the following bound using the Herbst argument (see \cite[p. 122]{BLM}).
\begin{theorem}\label{thm: main_thm}
Assume (A1) and (A2). There exist $C_2,C_3>0$ such that
\[
Ent~e^{\lambda T} \leq C_2\|x\|_1 \lambda^2 \mathbb{E}e^{\lambda T} \text{ for } -C_3\leq \lambda \leq 0, ~ x \in \mathbb{Z}^d\ .
\]
\end{theorem}
Theorem~\ref{thm: main_thm} will be proved in Section~\ref{sec: main_proof}.

\subsection{Properties of entropy}

We begin with some basic results on entropy. This material is taken from \cite[Section~2]{DHS1}, though it appears in various places, including \cite{BLM}. There is a variational characterization of entropy \cite[Section~5.2]{ledoux} that we will use.
\begin{proposition}\label{prop: variation}
We have the formula $Ent~X = \sup \bigl\{ \mathbb{E}XZ : \mathbb{E}e^Z \leq 1 \bigr\}$.
\end{proposition}
This characterization is typically used in the form: for any random variable $W$,
\begin{equation}\label{eq: exp_holder}
\mathbb{E}XW \leq Ent~X + \mathbb{E}X \log \mathbb{E}e^W.
\end{equation}

The second fact we need is a tensorization for entropy. For an edge $e$, write $Ent_e X$ for the entropy of $X$ considered as a function of only $t_e$ (with all other weights fixed).
\begin{proposition}\label{prop: tensor}
If $X \in L^2$ is a measurable function of $(t_e)$ then
\[
Ent~X \leq \sum_e \mathbb{E} Ent_e~X\ .
\]
\end{proposition}

\subsection{Application of Boucheron-Lugosi-Massart}
We say that $e \in \mathcal{E}^d$ is pivotal for $T$ if $e$ is in the intersection of all geodesics of $T$.
A pivotal edge has the property that increasing its weight, while leaving all other weights unchanged,
can change $T$ (and no other edges have this property).
Let $\mathrm{Piv}(0,x)$ be the set of all pivotal edges for $T$.
Our first bound toward Theorem~\ref{thm: main_thm} is the next proposition.
\begin{proposition}\label{prop: entropy_bound}
There exists a constant $C_4$ such that
\[
Ent~e^{\lambda T} \leq C_4 \lambda^2 \mathbb{E}\left[ e^{\lambda T} \# \mathrm{Piv}(0,x) \right]
\text{ for } \lambda \leq 0.
\]
\end{proposition}

Given a nonempty subset $S$ of $\mathcal{E}^d$, denote by $t_S$ a configuration $(t_e)_{e \in S}$
restricted to the edges in $S$.
Let $\{e_1, e_2, \ldots\}$ be an enumeration of $\mathcal{E}^d$.
Then, for each $e_i$, deterministically choose one of its endpoints $x_i$
and consider the box $B_i:=\{ v \in \mathbb{Z}^d;\|v-x_i\|_\infty \leq 1 \}$.
In addition, let $S_i$ be the edges between nearest-neighbor points of $B_i$:
\[
S_i :=
\bigl\{\{u,v\} \in \mathcal{E}^d : \|u-x_i\|_\infty \leq 1 \text{ and } \|v-x_i\|_\infty \leq 1
\bigr\}.
\]
Furthermore, set $f(x) := x\log x$ and
note that $f$ is convex on $[0,\infty)$ (with $f(0)=0$).

\begin{lemma}\label{lem:ent_box}
We have
\begin{align}\label{eq:ent_box}
 Ent~e^{\lambda T} \leq \sum_{i=1}^\infty \mathbb{E} Ent_{S_i}~e^{\lambda T},
\end{align}
where $Ent_{S_i}~e^{\lambda T}$ is the entropy of $e^{\lambda T}$
considered as a function of only $t_{S_i}$.
\end{lemma}
\begin{proof}

For each $i \geq 1$,
\[
\mathbb{E}Ent_{e_i} e^{\lambda T} = \iint f(e^{\lambda T})~\mathbb{P}(\text{d}t_{e_i})~\mathbb{P}(\text{d}t_{e_i^c}) - \int f\left( \int e^{\lambda T}~\mathbb{P}(\text{d}t_{e_i}) \right)~\mathbb{P}(\text{d}t_{e_i^c})\ .
\]
Here the measure $\mathbb{P}(\text{d}t_{e_i^c})$ represents the joint distribution of the edge-weights outside $e_i$.
Since $f$ is convex, Jensen's inequality shows that this is bounded above by
\[
\iint f(e^{\lambda T})~\mathbb{P}(\text{d}t_{S_i})~\mathbb{P}(\text{d}t_{S_i^c}) - \int f\left( \int e^{\lambda T}~\mathbb{P}(\text{d}t_{S_i}) \right)~\mathbb{P}(\text{d}t_{S_i^c}) = \mathbb{E}Ent_{S_i}~e^{\lambda T}
\]
Therefore Proposition~\ref{prop: tensor} implies that $Ent~e^{\lambda T}
 \leq \sum_{i=1}^\infty \mathbb{E} Ent_{e_i}~e^{\lambda T}
 \leq \sum_{i=1}^\infty \mathbb{E}Ent_{S_i}~e^{\lambda T}$.
\end{proof}

Let $T_i$ be the variable $T$ when the weights in $S_i$
are replaced by independent copies,
and let $\mathbb{E}_{S_i}'$ be an average over the variables in $S_i$ and their independent copies. The following lemma is nearly identical to \cite[Theorem~6.15]{BLM}, so we omit the proof.
\begin{lemma}\label{lem:yeah}
We have for $\lambda \leq 0$ and $i \geq 1$,
\begin{equation}\label{eq: yeah}
Ent_{S_i} e^{\lambda T} \leq \lambda^2 \mathbb{E}'_{S_i} \bigl[ e^{\lambda T}(T_i-T)_+^2 \bigr].
\end{equation}
\end{lemma}

Now we are in a position to prove Proposition~\ref{prop: entropy_bound}.

\begin{proof}[\bf Proof of Proposition~\ref{prop: entropy_bound}]
Lemmata~\ref{lem:ent_box} and \ref{lem:yeah} imply that
\begin{align}\label{eq:replace_bound}
 Ent~e^{\lambda T}
 \leq \lambda^2 \sum_{i=1}^\infty
      \mathbb{E} \Bigl[ \mathbb{E}'_{S_i} \bigl[ e^{\lambda T}(T_i-T)_+^2 \bigr] \Bigr].
\end{align}
To estimate the right side,
we first show that
\begin{align}\label{eq:difference}
 (T_i-T)_+ = (T_i-T)_+\mathbf{1}_{\{\mathrm{Piv}(0,x) \cap S_i \neq \emptyset\}}((t_e)).
\end{align}
We show that if $T_i >T$, then $\mathbf{1}_{\{\mathrm{Piv}(0,x) \cap S_i \neq \emptyset\}}((t_e)) =1$. So suppose that $\mathrm{Piv}(0,x)$ does not intersect $S_i$ in $(t_e)$; we must show that $(T_i-T)_+=0$. Under this assumption, there is a geodesic $\gamma$ in $(t_e)$ from $0$ to $x$ such that $\gamma$ does not contain edges in $S_i$, so in the configuration $(t_e')$, which equals $(t_e)$ off $S_i$ and equals the independent weights on $S_i$, the passage time of $\gamma$ is the same as in $(t_e)$.
This means that $T \geq T_i$ and \eqref{eq:difference} follows.

We next show that
\begin{equation}\label{eq: new}
(T_i-T)_+ \leq \sum_{\{u,v\} \in S_i} T_i(u,v),
\end{equation}
where $T_i(u,v)$ is the minimal passage time among all paths from $u$ to $v$
with edges in $S_i$ in the configuration $(t_e')$.
For $\{u,v\} \in S_i$, let $G_i(u,v)$ be a path with edges in $S_i$ from $u$ to $v$ with passage time in $(t_e')$ equal to $T_i(u,v)$. Let $\gamma$ be a (vertex self-avoiding) path from $0$ to $x$ with passage time in $(t_e)$ equal to $T$. Such a path exists by removing loops from a geodesic from $0$ to $x$. Define $\gamma'$ by replacing each of
its edges $\{u,v\} \in S_i$ by the path $G_i(u,v)$.
Then, the passage time of $\gamma'$ in the configuration $(t_e')$ is bounded above by $T+ \sum_{\{u,v\} \in \gamma \cap S_i} T_i(u,v) \leq T+ \sum_{\{u,v\} \in S_i} T_i(u,v)$,
and this shows \eqref{eq: new}.

By \eqref{eq:replace_bound}, \eqref{eq:difference} and \eqref{eq: new},
\begin{align*}
 Ent~e^{\lambda T}
 &\leq \lambda^2 \sum_{i=1}^\infty \mathbb{E} \Biggl[ \mathbb{E}'_{S_i} \Biggl[
       e^{\lambda T} \Biggl( \sum_{\{u,v\} \in S_i} T_i(u,v) \Biggr)^2
       \mathbf{1}_{\{\mathrm{Piv}(0,x) \cap S_i \neq \emptyset\}}((t_e)) \Biggr] \Biggr]\\
 &\leq \lambda^2 \mathbb{E} \Biggl( \sum_{\{u,v\} \in S_1} T_1(u,v) \Biggr)^2
       \mathbb{E} \bigl[ e^{\lambda T} C_4 \# \mathrm{Piv}(0,x) \bigr].
\end{align*}
for some constant $C_4$.
Note that, since $\mathbb{E}Y^2<\infty$ and there are at least $d$ (edge) disjoint paths between $u$ and $v$ with edges in $S_1$, the argument of Cox-Durrett \cite[Lemma~3.1]{CD} implies that
$\mathbb{E}T_1(u,v)^2<\infty$ for all $\{u,v\} \in S_1$.
Hence, Proposition~\ref{prop: entropy_bound} follows.
\end{proof}

\subsection{Proof of Theorem~\ref{thm: main_thm}}\label{sec: main_proof}
To prove Theorem~\ref{thm: main_thm}, we first show that
there exist $C_5$ and $C_6$ such that
for $\lambda \in (-C_5,0)$,
\begin{align}\label{eq:pre_main}
 \mathbb{E}\left[ e^{\lambda T} \#\mathrm{Piv}(0,x) \right] \leq C_6 \|x\|_1 \mathbb{E} e^{\lambda T}.
\end{align}

For the proof of \eqref{eq:pre_main}, for $c>0$ to be determined later,
we divide the left side of \eqref{eq:pre_main} into the following two parts:
\begin{equation}\label{eq: decomposition}
\mathbb{E} \left[ e^{\lambda T}\# \mathrm{Piv}(0,x) \right] \leq c\mathbb{E} \left[ e^{\lambda T}T \right] + \mathbb{E} \left[ e^{\lambda T} \# \mathrm{Piv}(0,x) \mathbf{1}_{\{ cT < \#\mathrm{Piv}(0,x)\}}\right]\ .
\end{equation}
Note that $T$ is an increasing function of $(t_e)$ whereas $e^{\lambda T}$ is decreasing (since $\lambda \leq 0$), so the Chebyshev association inequality \cite[Theorem~2.14]{BLM} gives
\begin{equation}\label{eq: stepone}
 \mathbb{E} \bigl[ e^{\lambda T} T \bigr]
 \leq \mathbb{E}e^{\lambda T} \mathbb{E}T
 \leq M \|x\|_1 \mathbb{E} e^{\lambda T},
\end{equation}
where $M:= \mathbb{E} \tau(0,e_1)$.
Therefore, the first term of \eqref{eq: decomposition} is harmless,
and we are left to deal with the second term.
To this end, we give the following proposition.
\begin{proposition}\label{prop: kesten_bound}
Assume (A2).
We can choose $c>0$ such that for some $\alpha,C_7>0$,
\[
\mathbb{E} e^{\alpha \phi_x} \leq C_7 \text{ for all } x\ ,
\]
where $\phi_x= \#\mathrm{Piv}(0,x) \mathbf{1}_{\{cT < \#\mathrm{Piv}(0,x)\}}$.
\end{proposition}
\begin{proof}
We will use a result of Kesten \cite[Proposition~5.8]{Kes86}, which says that if $\mathbb{P}(t_e=0)<p_c$ then there exist $a,C_8,C_9>0$ such that for all $m \in \mathbb{N}$,
\begin{align}\label{eq:kesten}
\mathbb{P}\bigl( \exists \text{ self-avoiding } \gamma \text{ starting at }0 \text{ with } \#\gamma \geq m \text{ but } \tau(\gamma) < am \bigr) \leq C_8e^{-C_9m}\ .
\end{align}
Let $A_m$ be the event that there is a self-avoiding path $\gamma$ starting at $0$ with $\#\gamma = m$ but $\tau(\gamma) < a \#\gamma$. $A_m$ implies the event in the left side in \eqref{eq:kesten}, so for all $m \geq 0$, $\mathbb{P}(A_m) \leq C_8e^{-C_9m}$. Denote by $A'_n$ the event that there is a self-avoiding path $\gamma$
starting at $0$ with $\#\gamma \geq n$ but $\tau(\gamma)<a\#\gamma$.
Then, there is a constant $C_{10}$ such that for $n \geq 0$,
\begin{align*}
 \mathbb{P}(A'_n)
 \leq \sum_{m=n}^\infty \mathbb{P}(A_m)
 \leq \sum_{m=n}^\infty C_8e^{-C_9m}
 \leq C_{10}e^{-C_9n},
\end{align*}
which implies that $\mathbb{P}(\phi_x \geq n) \leq \mathbb{P}(\{ \phi_x \geq n \} \cap (A'_n)^c)+C_{10}e^{-C_9n}$. On the event $\{ \phi_x \geq n \} \cap (A'_n)^c$ for $n \geq 1$, we have $cT<\# \mathrm{Piv}(0,x)$ and $\# \mathrm{Piv}(0,x) \geq n$,
so that, letting $\pi$ be a (self-avoiding) geodesic from $0$ to $x$, one has $\#\pi \geq \# \mathrm{Piv}(0,x) \geq n$, so
\[
T=\tau(\pi) \geq a\#\pi \geq a \# \mathrm{Piv}(0,x) > acT.
\] 
With these observations, if we choose $c:=a^{-1}$, then
$\mathbb{P}(\{ \phi_x \geq n \} \cap (A'_n)^c)$ is equal to zero.
This means that $\mathbb{P}(\phi_x \geq n) \leq C_{10}e^{-C_9n}$, and therefore,
by taking $\alpha:=C_9/2$, $\mathbb{E} e^{\alpha \phi_x}
 \leq C_{10} \sum_{m=1}^\infty e^{-C_9m/2}<\infty$, which proves the proposition.
\end{proof}

Return to the second term of \eqref{eq: decomposition}.
For the constant $c$ chosen in the proposition above, \eqref{eq: exp_holder} implies that the second term of \eqref{eq: decomposition}
is bounded above by
\[
\alpha^{-1}(Ent~e^{\lambda T} + \mathbb{E}e^{\lambda T} \log \mathbb{E}e^{\alpha \phi_x}) \leq \alpha^{-1}(Ent~e^{\lambda T} + (\log C_7) \mathbb{E}e^{\lambda T})\ .
\]
We put this back in \eqref{eq: decomposition} along with \eqref{eq: stepone} for
\[
\mathbb{E}\left[e^{\lambda T} \#\mathrm{Piv}(0,x) \right] \leq \alpha^{-1} Ent~e^{\lambda T} + (cM\|x\|_1 + \alpha^{-1} \log C_7) \mathbb{E}e^{\lambda T}\ .
\]
Placing this together with \eqref{eq: stepone} in \eqref{eq: decomposition} shows \eqref{eq:pre_main}. To finish the proof of Theorem~\ref{thm: main_thm}, combine \eqref{eq:pre_main} with Proposition~\ref{prop: entropy_bound} for
\[
\left( 1-\frac{C_4 \lambda^2}{\alpha}\right) Ent~e^{\lambda T} \leq C_4\lambda^2 \left( cM\|x\|_1+ \frac{\log C_7}{\alpha}\right)  \mathbb{E}e^{\lambda T}\ .
\]
If $-\sqrt{\frac{\alpha}{2C_4}} < \lambda < 0$, then one obtains the bound in Theorem~\ref{thm: main_thm}.

\section{Proof of Theorem~\ref{thm:shape}}\label{sec:shape}
We first show \eqref{eq:outer}
following the strategy of the proof of Theorem~3.1 of \cite[p.48]{Ale97}. The main difference here is the exponent on the $\log$ in \eqref{eq:restrict} and that Theorem~\ref{thm:big_one} replaces Eq.~(3.7) in \cite{Ale97}.

\begin{proof}[\bf Proof of (\ref{eq:outer}) in Theorem~\ref{thm:shape}]
We start by showing that there exists a constant $C_1$ such that with probability one,
\begin{align}
 B(n) \cap \mathbb{Z}^d \subset \{ n+C_1(n\log n)^{1/2} \} B_0
 \label{eq:restrict}
\end{align}
for all large $n$.
For $C>0$ and $n \in \mathbb{N}$, let $A_n^+(C)$ be the event that
there exists $y_n \in B(n) \cap \mathbb{Z}^d$ such that
$y_n \not\in \{ n+C(n\log n)^{1/2} \} B_0$.
At first, assume that $A_n^+(C)$ occurs.
Then $\tau (0,y_n) \leq n$ and $\mathbb{E} \tau (0,y_n) \geq \mu (y_n)>n+C(n\log n)^{1/2}$, so that
\begin{align}
 \tau (0,y_n)-\mathbb{E} \tau (0,y_n) \leq n-\mu (y_n)<-C(n\log n)^{1/2}.
 \label{eq:alex_shape}
\end{align}
In the case $\mu (y_n) \leq 2n$,
since $\mu(\cdot)$ is a norm on $\mathbb{R}^d$,
\eqref{eq:alex_shape} yields a constant $C_2$ such that
\begin{align*}
 \tau (0,y_n)-\mathbb{E} \tau (0,y_n)<-CC_2(\|y_n\|_1\log \|y_n\|_1 )^{1/2}.
\end{align*}
In the case $\mu (y_n)>2n$,
by the first inequality in \eqref{eq:alex_shape},
\begin{align*}
 \tau (0,y_n)-\mathbb{E} \tau (0,y_n) \leq -\mu (y_n)/2
 &<-CC_3(\|y_n\|_1 \log \|y_n\|_1)^{1/2}
\end{align*}
for some constant $C_3$.
With these observations,
on the event $\liminf_{n \to \infty}A_n^+(C)$,
\begin{align}
 \liminf_{y \to \infty} \frac{\tau (0,y)-\mathbb{E} \tau (0,y)}{(\|y\|_1\log \|y\|_1)^{1/2}}
 \leq -CC_4
 \label{eq:contradict}
\end{align}
where $C_4:=C_2 \wedge C_3$.
On the other hand, Theorem~\ref{thm:big_one} gives $C_5$ such that for $x \in \mathbb{Z}^d$,
\begin{align*}
 \mathbb{P} \left( \tau (0,x)-\mathbb{E} \tau(0,x) \leq -\frac{CC_4}{2}(\|x\|_1\log \|x\|_1)^{1/2} \right)
 \leq \|x\|_1^{-C^2C_4^2C_5/4}.
\end{align*}
Choose $C:=(8d)^{1/2}/(C_4C_5^{1/2})$.
By Borel--Cantelli, almost surely, $\tau(0,x) -\mathbb{E} \tau(0,x) > -\frac{CC_4}{2}(\|x\|_1\log \|x\|_1)^{1/2}$ for all large $x \in \mathbb{Z}^d$.
This together with \eqref{eq:contradict} implies \eqref{eq:restrict}.

To show \eqref{eq:outer},
note that $[-1/2,1/2]^d \subset (\lceil t \rceil \log \lceil t \rceil)^{1/2}B_0$ for sufficiently large $t$.
Since $B(t)$ is increasing in $t$, \eqref{eq:restrict} implies that
with probability one, for all large $t$,
\begin{align*}
 \frac{B(t)}{t}
 \subset \frac{\lceil t \rceil}{t} \frac{B(\lceil t \rceil)}{\lceil t \rceil}
 \subset \Bigl( 1+\frac{1}{t} \Bigr)
         \left\{ 1+(C_1+1)\lceil t \rceil^{-1/2} (\log \lceil t \rceil)^{1/2} \right\}B_0,
\end{align*}
and therefore \eqref{eq:outer} follows.
\end{proof}

Next we show \eqref{eq:inner}. We will use a result of Zhang \cite[Theorem~2]{Zha10}. For any two subsets $A$ and $B$ of $\mathbb{Z}^d$, define $\tau(A,B)$ as the minimal passage time of a path from a vertex in $A$ to one in $B$. If $\mathbb{E} t_e^\alpha<\infty$ for some $\alpha>1$, then for $k \geq 1$, \cite[Theorem~2]{Zha10} gives $C_6$ and $C_7$ (which may depend on $k$) such that
for all $\ell_1$-unit vectors $\xi \in \mathbb{R}^d$,
\begin{align}\label{eq:zhang}
\begin{split}
 \mathbb{P} \bigl( \bigl| \tau (D_m(0),D_m(m\xi ))-\mathbb{E} \tau (D_m(0),D_m(m\xi )) \bigr|
 \geq m^{1/2}(\log m)^4 \bigr)
 \leq C_6m^{-k},
\end{split}
\end{align}
where $D_m(v):=v+[-C_7(\log m)^2,C_7(\log m)^2]^d$ for $v \in \mathbb{Z}^d$.

\begin{proof}[\bf Proof of (\ref{eq:inner}) in Theorem~\ref{thm:shape}]
Assume that $\mathbb{E} t_e^\alpha<\infty$ for some $\alpha> 1+1/d$. Let
\begin{align*}
 B_0(n):=n \bigl\{ 1-Cn^{-1/2}(\log n)^4 \bigr\} B_0 \cap \mathbb{Z}^d \text{ for } n \geq 1,
\end{align*}
where $C$ is a constant to be chosen later.
By the same argument as in the last paragraph of the above proof,
it suffices to prove that with probability one,
\begin{align}\label{eq:discrete}
 B_0(n) \subset B(n)
\end{align}
for all large $n \in \mathbb{N}$.
A union bound gives
\begin{align*}
 \mathbb{P} (B_0(n) \not\subset B(n) \text{ for some } n \geq 1)
 \leq \sum_{n \geq 1} \sum_{v \in B_0(n)} \mathbb{P} (\tau(0,v) >n).
\end{align*}
We would like to show this sum is finite. Then Borel--Cantelli will finish the proof.

First we show that there exists $C_8>0$ such that
\begin{equation}\label{eq: first_case}
\sum_{n \geq 1} \sum_{\|v\|_1 \leq C_8 n} \mathbb{P}(\tau(0,v) > n) < \infty\ .
\end{equation}
For this, we use \cite[Theorem~4]{Ahl}. Let $Z$ be the minimum of $2d$ i.i.d. copies of $t_e$.
\begin{lemma}[Ahlberg]\label{lem: ahlberg}
Assume that $\mathbb{E}Z^\beta<\infty$ for some $\beta>0$. For every $\epsilon>0$ and $q \geq 1$ there exists $M=M(\beta,\epsilon,q)$ such that for every $z \in \mathbb{Z}^d$ and $s \geq \|z\|_1$,
\[
\mathbb{P}(\tau(0,z) - \mu(z) > \epsilon s) \leq M \mathbb{P}(Z \geq s/M)  + \frac{M}{s^q}\ .
\]
\end{lemma}
We will use this lemma with $\epsilon=1$ and $q=d+2$. It will also be used later in \eqref{eq: from_referee} with $q=2d+3$. Since $\mu(\cdot)$ is a norm on $\mathbb{R}^d$,
one can take a constant $C_9 \geq 1$ satisfying
\begin{equation}\label{eq: norm_equiv}
C_9^{-1} \|v\|_1 \leq \mu(v) \leq C_9 \|v\|_1 \text{ for all } v \in \mathbb{R}^d\ .
\end{equation}
If $\|v\|_1 \leq C_9^{-1} n/2$, then $\{\tau(0,v) > n\} \subset \{\tau(0,v)-\mu(v)>n/2\}$.
Lemma~\ref{lem: ahlberg} thus implies that for all $v \in \mathbb{Z}^d$ with $\| v \|_1 \leq C_9^{-1}n/2$,
\[
\mathbb{P}(\tau(0,v) > n) \leq M \mathbb{P}(Z \geq n/(2M)) + \frac{M}{(n/2)^{d+2}}\ .
\]
By choosing $C_8 = C_9^{-1}/2$ there is a constant $C_{10}$ such that
\[
\sum_{\|v\|_1 \leq C_8 n} \mathbb{P}(\tau(0,v) > n) \leq C_{10} n^d \left[ \mathbb{P}(Z \geq n/(2M)) + \frac{1}{n^{d+2}} \right],
\]
and the left side of \eqref{eq: first_case} is bounded by $C_{10} \sum_{n \geq 1} n^d \mathbb{P}(Z \geq n/(2M)) + C_{10} \sum_{n \geq 1} n^{-2}$.
This is finite so long as $\mathbb{E}Z^{d+1}<\infty$. To see why this holds, use Markov's inequality: $\mathbb{P}(Z \geq \lambda) = \left[ \mathbb{P}(t_e \geq \lambda) \right]^{2d} \leq \frac{\left( \mathbb{E}t_e^\alpha\right)^{2d}}{\lambda^{2\alpha d}}$, so there is $C_{11}$ such that
\[
\mathbb{E}Z^\beta \leq 1 + C_{11} \int_1^\infty x^{\beta-1-2\alpha d} ~\text{d}x < \infty \text{ whenever } \beta < 2\alpha d\ .
\]
Because $\alpha>1+1/d$, we obtain
\begin{equation}\label{eq: z_condition}
\mathbb{E}Z^{2d+2+\delta} <\infty \text{ for some } \delta > 0\ ,
\end{equation}
and so \eqref{eq: first_case} holds.

We move on to show that
\begin{equation}\label{eq: second_case}
\sum_{n \geq 1} \sum_{\stackrel{\|v\|_1 > C_8 n}{v \in B_0(n)}} \mathbb{P}(\tau(0,v) > n) < \infty\ .
\end{equation}
Given a nonzero $v \in B_0(n)$, we set $\xi:=v/\|v\|_1$ and $m:=\|v\|_1$.
To shorten the notation, denote $\tau_m:=\tau (D_m(0),D_m(m\xi))$.
Note that
\begin{align}\label{eq:box}
 \tau_m
 \leq \tau(0,v) \leq \tau_m +J_m(0) + J_m(m\xi)\ ,
\end{align}
where for $w \in \mathbb{Z}^d$, $J_m(w) = \max\{\tau(z_1,z_2) : z_1,z_2 \in D_m(w)\}$.
Because $v \in B_0(n)$, $n \geq \mu (\xi )m+Cn^{1/2}(\log n)^4$, and for some $C_{12}$ independent of $v,m$, $n \geq \mu(\xi)m + C_{12} Cm^{1/2}(\log m)^4$.
The second inequality of \eqref{eq:box} and Proposition~\ref{prop:fluct} give $C_{13}$ with
\begin{align*}
 \mathbb{P} (\tau (0,v)>n)
 &\leq \mathbb{P} ( \tau_m+J_m(0)+J_m(m\xi) >\mu(\xi)m + C_{12} Cm^{1/2}(\log m)^4)\\
 &\leq \mathbb{P} ( \tau_m+J_m(0)+J_m(m\xi)\\
 &\qquad \quad >\mathbb{E} \tau(0,v) -C_{13}(m \log m)^{1/2}+C_{12} Cm^{1/2}(\log m)^4).
\end{align*}
Due to the first inequality of \eqref{eq:box}, this is bounded by
\begin{align}\label{eq:lower_rate}
\begin{split}
 &\mathbb{P} ( \tau_m -\mathbb{E} \tau_m +J_m(0)+J_m(m\xi)
       > -C_{13}(m \log m)^{1/2}+C_{12} Cm^{1/2}(\log m)^4)\\
 &\leq \mathbb{P} \bigl( \tau_m-\mathbb{E} \tau_m
          >(C_{12} C/2)m^{1/2}(\log m)^4- C_{13} (m\log m)^{1/2} \bigr)\\
 &\quad
       +\mathbb{P} \bigl( J_m(0) + J_m(m\xi) >(C_{12} C/2)m^{1/2}(\log m)^4 \bigr).
\end{split}
\end{align}
For $C\geq 2C_{13} /(C_{12} (\log 2)^{7/2})+2$, by \eqref{eq:zhang} the second to last term is bounded above by
\begin{equation}\label{eq: ingredient_1}
\max_{\stackrel{v \in B_0(n)}{\|v\|_1 >C_8 n}} \mathbb{P}(\tau_m - \mathbb{E}\tau_m > (C_{12}C/2)m^{1/2}(\log m)^4 - C_{13} (m\log m)^{1/2}) \leq C_{14} n^{-(d+2)}\ .
\end{equation}
for some constant $C_{14}$.
As for the other term, write $\lambda_m = (C_{12}C/2)m^{1/2}(\log m)^4$ and estimate using subadditivity
\begin{align*}
\mathbb{P}(J_m(0)+J_m(m\xi) > \lambda_m) &\leq 4 \sum_{w \in D_m(0)} \mathbb{P}(\tau(0,w) \geq \lambda_m/4) \\
&\leq 4 \# D_m(0) \max_{w \in D_m(0)} \mathbb{P}(\tau(0,w) \geq \lambda_m/4)\ .
\end{align*}
We apply Lemma~\ref{lem: ahlberg} with $\epsilon=1$ and $q=2d+3$. If $C$ is large enough, then $\lambda_m/4 \geq 2 \mu(w)$ for all $w \in D_m(0)$. This gives $\mathbb{P}(\tau(0,w) \geq \lambda_m/4) \leq \mathbb{P}(\tau(0,w) - \mu(w) \geq \lambda_m/8)$. Again increasing $C$, so that $\lambda_m/8 \geq \max_{w \in D_m(0)} \|w\|_1$ for all $m$, use Lemma~\ref{lem: ahlberg} with $s = \frac{\lambda_m}{8} \geq \|w\|_1$ for
\begin{equation}\label{eq: from_referee}
\mathbb{P}(\tau(0,w) \geq \lambda_m/4) \leq M\mathbb{P}(Z \geq \lambda_m/(8M)) + \frac{M}{(\lambda_m/8)^{2d+3}}\ .
\end{equation}
This shows that there are constants $C_{15}$ and $C_{16}$ such that
\begin{align*}
\max_{\stackrel{v \in B_0(n)}{\|v\|_1 > C_8 n}} \mathbb{P}(J_m(0) + J_m(m\xi) > \lambda_m)
&\leq C_{15} n^{\delta/2}\left[ \mathbb{P}(Z \geq C_{16} \sqrt n) + n^{-(d+3/2)}\right]\ ,
\end{align*}
where $\delta \in (0,1)$ is from \eqref{eq: z_condition} is used to bound $\# D_m(0)$. Combining this with \eqref{eq: ingredient_1} and placing them into \eqref{eq:lower_rate} implies
\[
\sum_{n \geq 1} \sum_{\stackrel{\|v\|_1>C_8 n}{v \in B_0(n)}} \mathbb{P}(\tau(0,v)>n) \leq \sum_{n \geq 1} C_{17} n^{d+\delta/2} \left[ \mathbb{P}(Z \geq C_{16} \sqrt n) + n^{-(d+3/2)} \right]
\]
for some constant $C_{17}$.
Because $\delta<1$, this sum converges if and only if $\sum_{n \geq 1} n^{d+\delta/2} \mathbb{P}(Z \geq C_{16} \sqrt n)$ does, which is the same as $\sum_{n \geq 1} n^{d+\delta/2} \mathbb{P}(Z^2 \geq C_{16}^2 n) < \infty$. Because $\mathbb{E}Z^{2d+2+\delta}<\infty$ from \eqref{eq: z_condition}, we are done.
\end{proof}



\begin{thebibliography}{99}
\footnotesize

\bibitem{Ahl}
D.~Ahlberg.
\newblock A {H}su-{R}obbins-{E}rd\"os strong law in first-passage percolation.
\newblock {\em The Annals of Probability}, to appear.
\newblock arXiv:1305.6260.

\bibitem{Ale93}
K.~S. Alexander.
\newblock A note on some rates of convergence in first-passage percolation.
\newblock {\em The Annals of Applied Probability}, 3(1):81--90, 1993.

\bibitem{Ale97}
K.~S. Alexander.
\newblock Approximation of subadditive functions and convergence rates in
  limiting-shape results.
\newblock {\em The Annals of Probability}, 25(1):30--55, 1997.

\bibitem{AZ}
K.~S. Alexander and N.~Zygouras.
\newblock Subgaussian concentration and rates of convergence in directed
  polymers.
\newblock {\em Electronic Journal of Probabability}, 18(5):1--28, 2013.

\bibitem{ADdiff}
A.~Auffinger and M.~Damron.
 \newblock Differentiability at the edge of the percolation cone and related results 
  in first-passage percolation.
  \newblock {\em Probability Theory and Related Fields}, 156:193--227.

\bibitem{BR}
M.~Bena\"im and R.~Rossignol.
\newblock Exponential concentration for first passage percolation through
  modified Poincar{\'e} inequalities.
\newblock {\em Annales de l'Institut Henri Poincar\'e Probabilit\'es et Statistiques}, 44, 544-573, 2008.

\bibitem{BLM}
S.~Boucheron, G.~Lugosi, and P.~Massart.
\newblock {\em Concentration Inequalities: A Nonasymptotic Theory of
  Independence}.
\newblock Oxford University Press, 2013.

\bibitem{CDey}
S.~Chatterjee and P.~S. Dey.
\newblock Central limit theorem for first-passage percolation time across thin
  cylinders.
\newblock {\em Probability Theory and Related Fields}, 156(3-4):613--663, 2013.

\bibitem{CD}
J.~T. Cox and R.~Durrett.
\newblock Some limit theorems for percolation processes with necessary and
  sufficient conditions.
\newblock {\em The Annals of Probability}, pages 583--603, 1981.

\bibitem{DHS1}
M.~Damron, J.~Hanson, and P.~Sosoe.
\newblock Sublinear variance in first-passage percolation for general distributions.
\newblock {\em Probability Theory and Related Fields}, 163(1):223--258, 2015.

\bibitem{DHS}
M.~Damron, J.~Hanson, and P.~Sosoe.
\newblock Subdiffusive concentration in first-passage percolation.
\newblock {\em Electronic Journal of Probability}, 19(109):1--27, 2014.

\bibitem{DK_arxiv}
M.~Damron and N.~Kubota.
\newblock Gaussian concentration for the lower tail in first-passage percolation under low moments.
\newblock {\em arXiv}:1406.3105.

\bibitem{Kes86}
H.~Kesten.
\newblock Aspects of first passage percolation.
\newblock In {\em \'{E}cole d'\'et\'e de probabilit\'es de {S}aint-{F}lour,
  {XIV}---1984}, volume 1180 of {\em Lecture Notes in Math.}, pages 125--264.
  Springer, Berlin, 1986.

\bibitem{Kes93}
H.~Kesten.
\newblock On the speed of convergence in first-passage percolation.
\newblock {\em The Annals of Applied Probability}, pages 296--338, 1993.

\bibitem{kubota}
N.~Kubota.
\newblock Upper bounds on the non-random fluctuations in first passage
  percolation with low moment conditions.
\newblock {\em Yokohama Mathematical Journal}, to appear.
\newblock arXiv:1306.5917v3.

\bibitem{ledoux}
M.~Ledoux.
\newblock {\em The concentration of measure phenomenon}, volume~89.
\newblock AMS Bookstore, 2005.

\bibitem{rhee}
W.~T. Rhee.
\newblock On rates of convergence for common subsequences and first passage
  time.
\newblock {\em Annals of Applied Probability}, 5(1):44--48, 1995.

\bibitem{Talagrand}
M.~Talagrand.
\newblock Concentration of measure and isoperimetric inequalities in product
  spaces.
\newblock {\em Publications Math{\'e}matiques de l'Institut des Hautes Etudes
  Scientifiques}, 81(1):73--205, 1995.

\bibitem{Zha10}
Y.~Zhang.
\newblock On the concentration and the convergence rate with a moment condition
  in first passage percolation.
\newblock {\em Stochastic Processes and their Applications}, 120(7):1317--1341,
  2010.

\end{thebibliography}

\bigskip
\noindent
{\bf Acknowledgments.}

We thank an anonymous reviewer for pointing out Talagrand's lower tail concentration inequality under two moments. The research of M. D. is supported by NSF grant DMS-1419230.


%
%
%
%

\end{document}